\documentclass[11pt]{amsart}

\def\R{{\mathbb {R}}}
\def\N{{\mathbb {N}}}

\def\A{{\mathcal{A}}}
\def\PP{{\mathcal{P}}}

\def\ve{\varepsilon}

\def\sop{\operatorname {\text{supp}}}

\def\diam{\operatorname {\text{diam}}}
\def\div{\operatorname {\text{div}}}

\def\cp{\text{\rm{cap}}}

\newtheorem{teo}{Theorem}[section]
\newtheorem{lema}[teo]{Lemma}
\newtheorem{prop}[teo]{Proposition}
\newtheorem{corol}[teo]{Corollary}

\theoremstyle{remark}
\newtheorem{remark}[teo]{Remark}

\theoremstyle{definition}
\newtheorem{defi}[teo]{Definition}

\numberwithin{equation}{section}

\begin{document}

\title{An extension of a Theorem of V. \v{S}ver\'ak to variable exponent spaces}

\author[C. Baroncini, J. Fern\'andez Bonder]{Carla Baroncini and Juli\'an Fern\'andez Bonder}

\address{IMAS - CONICET and Departamento de Matem\'atica, FCEyN - Universidad de Buenos Aires, Ciudad Universitaria, Pabell\'on I  (1428) Buenos Aires, Argentina.}

\email[J. Fernandez Bonder]{jfbonder@dm.uba.ar}

\urladdr[J. Fernandez Bonder]{http://mate.dm.uba.ar/~jfbonder}

\email[C. Baroncini]{cbaroncin@dm.uba.ar}

%49J45  	Methods involving semicontinuity and convergence; relaxation
%49Q10  	Optimization of shapes other than minimal surfaces


\begin{abstract}
In 1993, V. \v{S}ver\'ak proved that if a sequence of uniformly bounded domains $\Omega_n\subset \R^2$ such that $\Omega_n\to \Omega$ in the sense of the Hausdorff complementary topology, verify that the number of connected components of its complements are bounded, then the solutions of the Dirichlet problem for the Laplacian with source $f\in L^2(\R^2)$ converges to the solution of the limit domain with same source. In this paper, we extend \v{S}ver\'ak result to variable exponent spaces.
\end{abstract}

\maketitle

\section{Introduction}
One important problem in partial differential equations is the stability of solutions with respect to perturbations on the domain. This problem has fundamental applications in numerical computations of the solutions and is also fundamental in optimal shape design problems. See \cite{Allaire, Henrot, Pironneau} and references therein.

The famous example of Cioranescu and Murat \cite{C-M} shows that this problem presents severe difficulties when treated in full generality. In fact, in \cite{C-M} the authors take $D=[0,1]\times [0,1]\subset \R^2$ and define  the domains $\Omega_n = D\setminus \cup_{i,j=1}^{n-1} B_{r_n}(x_{i,j}^n)$ where the centers of the balls $x_{i,j}^n = (i/n, j/n)$, $1\le i,j\le n-1$ and the radius $r_n = n^{-2}$. Then these domains $\Omega_n$ converge to the empty set in the Hausdorff complementary topology, but if $u_n\in H^1_0(\Omega_n)$ is the solution to
$$
\begin{cases}
-\Delta u_n = f & \text{in }\Omega_n,\\
u_n=0 & \text{on }\partial\Omega_n,
\end{cases}
$$
then $u_n\rightharpoonup u^*$ weakly in $H^1_0(D)$ to the solution of
$$
\begin{cases}
-\Delta u^* + \frac{2}{\pi} u^* = f & \text{in }D,\\
u^*=0 & \text{on }\partial D.
\end{cases}
$$

This example can be generalized to other space dimensions, to different bounded sets $D$ and also to different types of {\em holes}. See the original work \cite{C-M} and also \cite{Tartar}.

There are some simple cases where the continuity can be granted. For instance, if $\Omega$ is convex and $\{\Omega_n\}_{n\in\N}$ is an increasing sequence of convex polygons such that $\Omega = \cup_{n\in \N}\Omega_n$, then the solutions of the approximating domains $\Omega_n$ converges to the one of $\Omega$. This fact can be traced back to the late 50's and the beginning of the 60's, see \cite{Babuska, Hong57, Hong58, Hong59}. Then, this result can be generalized in terms of the capacity of the symmetric differences of $\Omega$ and $\Omega_n$. See the book of Henrot, \cite{Henrot}.

In practical applications, when one does not have control on the sequence of approximating domains, this hypothesis is uncheckable,  so a different condition is needed. \v{S}ver\'ak  in \cite{Sverak} gave such a condition. In fact, given a bounded domain $D\subset \R^2$ and a sequence of domains $\Omega_n\subset D$ such that $\Omega_n\to \Omega$ in the sense of the Hausdorff complementary topology the condition that guaranty the convergence of the solutions in $\Omega_n$ to the one in $\Omega$ is that the number of connected components of $D\setminus\Omega_n$ be bounded. c.f. with the example of Cioranescu-Murat.

The reason why \v{S}ver\'ak's result holds in dimension 2 is because the capacity of curves in dimension 2 is positive, while in higher dimension curves have zero capacity.

\v{S}ver\'ak's result was later generalized to nonlinear elliptic equations of $p-$Laplace type. In fact, in \cite{Bucur}, the authors prove the continuity of the solutions of
$$
\begin{cases}
-\Delta_p u_n = f & \text{in }\Omega_n\subset \R^N, \\
u_n = 0 & \text{on }\partial\Omega_n,
\end{cases}
$$
when the domains $\Omega_n$ converges to $\Omega$ in the Hausdorff complementary topology under the assumption that the number of connected components of its complements remains bounded. The idea of the proof is similar to the original one of \v{S}ver\'ak and so they end up with the restriction $p>N-1$ that is needed for the curves to have positive $p-$capacity.

Recall that $\Delta_p u = \div(|\nabla u|^{p-2} \nabla u)$ is the so-called $p-$laplace operator.

In recent years a lot of attention have been put in nonlinear elliptic equations with nonstandard growth. One of the most representative of such equations is the so-called $p(x)-$laplacian, that is defined as $\Delta_{p(x)} u = \div(|\nabla u|^{p(x)-2}\nabla u)$. This operator  became very popular due to many new interesting applications, for instance in the mathematical modeling of electrorheological fluids (see \cite{Ru}) and also in image processing (see \cite{CLR}). Here, the exponent $p(x)$ is assumed to be measurable and bounded away from 1 and infinity.

So, the purpose of this paper is the extension of the result of \v{S}ver\'ak (and also the results of \cite{Bucur}) to the variable exponent setting.

\subsection*{Organization of the paper} The rest of the paper is organized as follows. In section 2 we collect some preliminaries on variable exponent spaces that are needed in this paper. The standard reference for this is the book \cite{Diening}. Some results are slight variations of the ones found in \cite{Diening} and in these cases we present full proofs of those facts (c.f. Theorem \ref{teocaracterizacion}).

In section 3, we study the Dirichlet problem for the $p(x)-$laplacian, the main result being the continuity of the solution with respect to the source. Although some of the results are well known, we decided to present the proofs of all of the results since we were unable to find a reference for these.

In section 4 we analyze the dependence of the solution of the Dirichlet problem for the $p(x)-$laplacian with respect to variations on the domain. Our two main theorems here are Theorem \ref{teo2dirichlet} where a capacity condition on the sequence of approximating domains is given in order for the continuity of solutions to hold, and Theorem \ref{teoindep} where it is shown that the continuity only depends on the approximating domains and not on the source term.

In section 5 after giving some capacity estimates that are needed in the remaining of the paper, collect all of our results and prove the main result of the paper, namely the extension of \v{S}ver\'ak's result to the variable exponent setting, i.e. Theorem \ref{teosverak}.

\section{Preliminaries}

\subsection{Definitions and well-known results}

Given $\Omega\subset \R^N$ a bounded open set, we consider the class of exponents $\PP(\Omega)$ given by
$$
\PP(\Omega) := \{p\colon \Omega\to [1,\infty)\colon p \text{ is measurable and bounded}\}.
$$

The variable exponent Lebesgue space $L^{p(x)}(\Omega)$ is defined by
$$
L^{p(x)}(\Omega):= \Big\{f\in L^1_{\text{loc}}(\Omega)\colon \rho_{p(x)}(f)<\infty\Big\},
$$
where the modular $\rho_{p(x)}$ is given by
$$
\rho_{p(x)}(f) := \int_{\Omega} |f|^{p(x)}\, dx.
$$
This space is endowed with the luxemburg norm
$$
\|f\|_{L^{p(x)}(\Omega)} = \|f\|_{p(x),\Omega} = \|f\|_{p(x)} := \sup\Big\{\lambda>0\colon \rho_{p(x)}(\tfrac{f}{\lambda})<1\Big\}. 
$$

The infimum and the supremum of the exponent $p$ play an important role in the estimates as the next elementary proposition shows. For further references, the following notation will be imposed
$$
1\le p_-:= \inf_{\Omega}p \le \sup_{\Omega} p =: p_+<\infty.
$$
The proof of the following proposition can be found in \cite[Theorem 1.3, p.p. 427]{FanyZhao}.

\begin{prop}\label{propdesigualdades}
Let $f\in L^{p(x)}(\Omega)$, then
$$
\min\{\|f\|_{p(x)}^{p_-}, \|f\|_{p(x)}^{p_+}\} \leq \rho_{p(x)}(f)\leq \max\{\|f\|_{p(x)}^{p_-}, \|f\|_{p(x)}^{p_{+}}\}.
$$
\end{prop}

\begin{remark}\label{minmax}
Proposition \ref{propdesigualdades}, is equivalent to
$$
\min\{\rho_{p(x)}(f)^{\frac{1}{p_-}}, \rho_{p(x)}(f)^{\frac{1}{p_{+}}} \} \leq \|f\|_{p(x)}\leq \max\{\rho_{p(x)}(f)^{\frac{1}{p_-}}, \rho_{p(x)}(f)^{\frac{1}{p_{+}}} \}.
$$
\end{remark}

We will use the following form of H\"older's inequality for variable exponents. The proof, which is an easy consequence of Young's inequality, can be found in \cite[Lemma 3.2.20]{Diening}.

\begin{prop}[H\"older's inequality]\label{propholder}
Assume $p_->1$. Let $u\in L^{p(x)}(\Omega)$ and $v\in L^{p'(x)}(\Omega)$, then
$$
\int_{\Omega} |u v|\, dx\leq 2\|u\|_{p(x)} \|v\|_{p'(x)},
$$
where $p'(x)$ is, as usual, the conjugate exponent, i.e. $p'(x):= p(x)/(p(x)-1)$.
\end{prop}

The variable exponent Sobolev space $W^{1,p(x)}$ is defined by
$$
W^{1,p(x)}(\Omega):=\Big\{u\in W^{1,1}_{\text{loc}}(\Omega)\colon u\in L^{p(x)}(\Omega) \text{ and } \partial_i u\in L^{p(x)}(\Omega)\ i=1,\dots,N\Big\},
$$
where $\partial_i u$ stands fot the $i-$th partial weak derivative of $u$.

This space posses a natural modular given by
$$
\rho_{1,p(x)}(u) := \int_\Omega |u|^{p(x)} + |\nabla u|^{p(x)}\, dx,
$$
so $u\in W^{1,p(x)}(\Omega)$ if and only if $\rho_{1,p(x)}(u)<\infty$.

The corresponding luxenburg norm associated to this modular is
$$
\|u\|_{W^{1,p(x)}(\Omega)} = \|u\|_{1,p(x),\Omega} = \|u\|_{1,p(x)} := \sup\Big\{\lambda>0\colon \rho_{1,p(x)}(\tfrac{u}{\lambda})<1\Big\}. 
$$
Observe that this norm turns out to be equivalent to $\|u\|:= \|u\|_{p(x)} + \|\nabla u\|_{p(x)}$.

One important subspace of $W^{1,p(x)}(\Omega)$ is the functions with zero boundary values. This is the content of the next definition.
\begin{defi}
We define $W^{1,p(x)}_{0}(\Omega)$ as the closure in $W^{1,p(x)}(\Omega)$ of functions with compact support.
\end{defi}

In most applications is very helpful to have test functions to be dense in $W^{1,p(x)}_0(\Omega)$. It is well known, see \cite{Diening}, that this property fails in general, even for continuous exponents $p(x)$. In order to have this desired property one need to impose some regularity conditions on the exponent $p(x)$. 

\begin{defi}We say that $p \colon \Omega\rightarrow {\R}$ is {\em log-H\"older continuous} in $\Omega$ if 
\begin{equation}\label{logholder}
\sup_{\substack{x,y\in \Omega\\ x\neq y}}\log(|x-y|^{-1})|p(x)-p(y)| < \infty.
\end{equation}
Set $\PP^{log}(\Omega)=\{p \in \PP(\Omega) \colon p \text{ satisfies \eqref{logholder}}\}$.
\end{defi}

Under this condition, the following theorem holds,
\begin{teo}[Theorem 9.1.6 in \cite{Diening}]
Assume that $p \in \mathcal P^{log}(\Omega)$, then $C^{\infty}_{c}(\Omega)$ is dense in $W^{1,p(x)}_0(\Omega)$.
\end{teo}

The proof of the following theorem can be found in \cite[Theorem 8.2.4]{Diening}.

\begin{teo}[Poincar\'e's inequality.]\label{teopoincare}
Let $p \in \PP^{log}(\Omega)$. Then there exists a constant $c>0$ such that
$$
\|u\|_{p(x)} \leq c \|\nabla u\|_{p(x)},\quad u\in W^{1,p(x)}_0(\Omega).
$$
\end{teo}

\begin{remark}
Thanks to Poincar\'e inequality, as usual, in $W^{1,p(x)}_0(\Omega)$ the following norm will be used,
$$
\|u\|_{W^{1,p(x)}_0(\Omega)} = \|\nabla u\|_{p(x)}.
$$
This norm, is equivalent to the usual norm in $W^{1,p(x)}(\Omega)$ for functions $u\in W^{1,p(x)}_0(\Omega)$.
\end{remark}

\begin{defi}
We denote by $W^{-1, p'(x)}(\Omega)$ the topological dual space of $W^{1,p(x)}_0(\Omega)$.

The duality product between $f\in W^{-1, p'(x)}(\Omega)$ and $u\in W^{1, p(x)}_0(\Omega)$ will be denoted, as usual, by $\langle f, u\rangle$.

The norm in this space will be denoted by
$$
\|f\|_{W^{-1,p'(x)}(\Omega)} = \|f\|_{-1,p'(x)} := \sup\{\langle f, u\rangle \colon u\in W^{1,p(x)}_0(\Omega),\ \|\nabla u\|_{p(x)}\le 1\}.
$$
\end{defi}

We now present a result which we will find most useful later. 
\begin{prop} \label{propdensidad}
The space $L^{\infty}(\Omega)$ is dense in $W^{-1,p'(x)}(\Omega)$.
\end{prop}

\begin{proof} 
By H\"older's inequality we have that $W^{1,p_{+}}_{0}(\Omega)\subset W^{1,p(x)}_{0}(\Omega) \subset W^{1,p_-}_{0}(\Omega)$ with continuous embeddings. Since $C^\infty_c(\Omega)\subset W^{1,p_{+}}_{0}(\Omega)$ and $p\in \PP^{log}(\Omega)$ we have the embeddings are dense.
Therefore,
$$
W^{-1,(p_-)'}(\Omega) \subset W^{-1,p'(x)}(\Omega) \subset W^{-1,(p_{+})'}(\Omega),
$$
with dense embeddings. Finally, since $L^{\infty}(\Omega)$ is dense in $W^{-1,(p_{-})'}(\Omega)$, we have that $L^{\infty}(\Omega)$ is dense in $W^{-1,p'(x)}(\Omega)$.
\end{proof}

Analogous to the constant exponent case, we have the following characterization of $W^{-1, p'(x)}(\Omega)$.
\begin{prop}
Let $f\in W^{-1, p'(x)}(\Omega)$. Then, there exists $\{f_i\}_{i=0}^N\subset L^{p'(x)}(\Omega)$ such that
$$
\langle f, u\rangle = \int_\Omega f_0 u\, dx - \sum_{i=1}^N \int_{\Omega} f_i\partial_i u\, dx.
$$
We will then say that $f = f_0 + \sum_{i=1}^N \partial_i f_i$. Moreover, 
$$
\|f\|_{*} = \inf\left\{\sum_{i=0}^N \|f_i\|_{p'(x)}\colon f = f_0 + \sum_{i=1}^N \partial_i f_i,\ f_i\in L^{p'(x)}(\Omega), i=0,\dots,N\right\},
$$
defines an equivalent norm in $W^{-1,p'(x)}(\Omega)$.
\end{prop}

\begin{proof}
The characterization of $W^{-1,p'(x)}(\Omega)$ follows exactly as in the constant exponent case. It remains to see the equivalence of the norms $\|\cdot\|_{-1,p'(x)}$ and $\|\cdot\|_*$.

Observe that $\|\cdot\|_{*}$ clearly defines a norm in $W^{-1,p'(x)}(\Omega)$.

Let us now take $f_0, f_{1},\dots,f_{n}\in L^{p'(x)}(\Omega)$ such that $f= f_0 + \sum_{i=1}^N \partial_i f_i$ and consider $v \in W_{0}^{1,p(x)}(\Omega)$ such that $\|\nabla v\|_{p(x)}=1$. 

By H\"older's inequality (Proposition \ref{propholder}) and Poincar\'e's inequality (Theorem \ref{teopoincare}), we have 
\begin{align*}
\langle f,v \rangle &=\int_{\Omega}\left(f_{0}v+\sum_{i=1}^{N} f_{i}\partial_{i}v\right)\, dx\\
& \leq 2 \|f_0\|_{p'(x)}\|v\|_{p(x)} + 2 \sum_{i=1}^N \|f_i\|_{p'(x)}\|\partial_i v\|_{p(x)}\\
&\leq C \left(\|f_0\|_{p'(x)}+\sum_{i=1}^N \|f_i\|_{p'(x)}\right).
\end{align*}
Therefore,
$$
\|f\|_{-1,p'(x)}=\inf_{\|\nabla v\|_{p(x)}=1}\langle f,v \rangle \leq  C \left(\|f_0\|_{p'(x)}+\sum_{i=1}^N \|f_i\|_{p'(x)}\right),
$$
so 
$$
\|f\|_{-1,p'(x)}\le C\|f\|_*
$$
Now, the reverse inequality is a direct consequence of the Open Mapping Theorem (cf. \cite{Brezis}).
\end{proof}

\begin{remark}
Let now $D\subset \R^N$ be a bounded, open set and let $\Omega\subset D$ be open. Then, we have that $W^{1,p(x)}_0(\Omega)\subset W^{1,p(x)}_0(D)$, the inclusion being canonical, extending by zero. This inclusion induces $W^{-1, p'(x)}(D)\subset W^{-1,p'(x)}(\Omega)$ by restriction. Therefore, when dealing with sets $\Omega$ that are subsets of $D$, if one is considering $f\in W^{-1,p'(x)}(D)$ and $u\in W^{1,p(x)}_0(\Omega)$ there is no ambiguity in the notation $\langle f, u\rangle$.
\end{remark}

\subsection{$p(x)$-capacity and pointwise properties of Sobolev functions} We need the concept of capacity modified to deal with pointwise properties of functions in $W^{1,p(x)}_0(\Omega)$. This is the concept of $p(x)-$capacity. See \cite[Chapter 10]{Diening}. 

\begin{defi}
Given $E\subset \mathbb{R}^{N}$, we consider the set
$$
S_{p(x)}(E)=\left\{u\in W^{1,p(x)}(\mathbb{R}^{N})\colon u\geq 0 \text{ and } u\geq 1 \text{ in an open set containing } E \right\}.
$$
If $S_{p(x)}(E)\neq \emptyset$, we define {\em $p(x)-$Sobolev capacity} of E as follows 
$$
\cp_{p(x)}(E)=\inf_{u\in S_{p(x)}(E)}\int_{\mathbb{R}^N} |u|^{p(x)}+|\nabla u|^{p(x)}dx = \inf_{u\in S_{p(x)}(E)} \rho_{1,p(x)}(u).
$$
If $S_{p(x)}(E)=\emptyset$, we set $\cp_{p(x)}(E)=\infty$. 
\end{defi}

\begin{defi}
Let $p \in \PP^{log}(\Omega)$ and $K \subset \Omega$ compact, we define the \textit{$p(x)-$relative capacity} as
$$
\cp^{*}_{p(x)}(K,\Omega)=\inf_{u \in R_{p(x)}(K,\Omega)} \rho_{p(x),\Omega}(\nabla u)
$$
where $R_{p(x)}(K,\Omega)=\{u \in W^{1,p(x)}_0(\Omega)\colon u>1 \text{ in } K \text{ and } u\geq 0\}$.

If $U \subset \Omega$ is an open set, we define $\cp_{p(x)}(U,\Omega)=\displaystyle{\sup_{\substack {K \subset U\\ K \text{ compact}}} \cp^{*}_{p(x)}(K,\Omega)}$.

Finally, if $E \subset \Omega$ is arbitrary, we define the $p(x)-$ relative capacity of $E$ with respect to $\Omega$ as
$$
\cp_{p(x)}(E,\Omega)=\inf_{\substack{E \subset U \subset \Omega\\ U \text{ open}}} \cp_{p(x)}(U,\Omega).
$$
\end{defi}

The main advantage of the relativa capacity is the fact that is possible to obtain a {\em capacitary potential}, i.e. a function whose modular gives the capacity of a set.

To this end, let $A\subset D$ and consider the class
$$
\Gamma_{A}=\overline{\left\{v\in W_{0}^{1,p(x)}(D) \colon v\geq 1 \text{ a.e. in an open set containing } A\right\}},
$$
the closure being taken in $W^{1,p(x)}_0(D)$.

\begin{remark}\label{wclosed}
Observe that since $\Gamma_{A}\subset W_{0}^{1,p(x)}(D)$ is closed and convex (the closure of a convex set is convex), it follows that is weakly convex. This fact will be used in the next proposition.
\end{remark}

Now we show that the relative capacity of a set is realized by a function in $\Gamma_A$.
\begin{prop}
If $\Gamma_{A}\neq \emptyset$, then there exists a unique $u_{A}\in \Gamma_{A}$ such that
$$
\cp_{p(x)}(A,D)=\int_{D}\left|\nabla u_{A}\right|^{p(x)}dx.
$$
\end{prop}

\begin{proof}
Consider $\{v_n\}_{n\in\N}\subset W_{0}^{1,p(x)}(D)$ such that $v_{n}\geq 1$ a.e. in an open set containing $A$ and 
$$
\int_{D}\left|\nabla v_{n}\right|^{p(x)}\, dx \to \cp_{p(x)}(A,D).
$$

By Theorem \ref{teopoincare} and Proposition \ref{propdesigualdades}, we have
$$
\|\nabla v_{n}\|_{p(x)}\leq \max \{\rho_{p(x)}(\nabla v_{n})^{\frac{1}{p_+}}, \rho_{p(x)}(\nabla v_{n})^{\frac{1}{p_-}}\}.
$$
Then, $\{v_n\}_{n\in\N}$ is bounded in $W_{0}^{1,p(x)}(D)$, which is a reflexive space. By Alaoglu's Theorem, there is a subsequence $v_{n_{j}}\rightharpoonup v_{\infty}$ en $W_{0}^{1,p(x)}(D)$. By Remark \ref{wclosed}, $v_{\infty}\in \Gamma_{A}$.

Observe that
$$
\int_{D}\left|\nabla v_{\infty}\right|^{p(x)}dx\leq \liminf \int_{D}\left|\nabla v_{n_{j}}\right|^{p(x)}dx=\cp_{p(x)}(A,D).
$$
Since the reverse inequality is obvious, the first part of the Proposition is proved.

The uniqueness is an immediate consequence of the strict convexity of the modular, since $p_->1$. We leave the details to the reader.
\end{proof}

We can now give the definition of capacitary potential.
\begin{defi}
We define the capacitary potential of $A$ such as the only $u_{A}$ that verifies
$$
\int_{D}\left|\nabla u_{A}\right|^{p(x)}dx=\inf_{ v \in \Gamma_{A}} \int_{D}\left|\nabla v\right|^{p(x)}dx = \cp_{p(x)}(A,D).
$$ 
\end{defi}

It is well known that when dealing with pointwise properties of Sobolev functions, the concept of {\em almost everywhere} needs to be changed to {\em quasi everywhere}. This is the content of the next definition.

\begin{defi}
An statement is valid {\em $p(x)-$quasi everywhere} ($p(x)-$q.e.) if it is valid except  in a set of null Sobolev $p(x)-$capacity.
\end{defi}

\begin{defi}
Let $D \subset {\R}^{N}$ be an open bounded set, $\Omega \subset D$ is {\em $p(x)-$quasi open} if there is a decreasing sequence $\{W_{n}\}_{n\in\N}$ of open sets such that $\cp_{p(x)}(W_{n},D)$ converges to $0$ and $\Omega\cup W_{n}$ is an open set for each $n$.
\end{defi}

\begin{defi}
A function $u \colon \Omega \rightarrow {\R}$ is {\em $p(x)-$quasi continuous} if for every $\ve>0$, there is an open set $U$ such that $\cp_{p(x)}(U)<\ve$ and $u|_{\Omega \setminus U}$ is continuous.
\end{defi}

The proof of the next theorem can be found in \cite[Corollary 11.1.5]{Diening}.

\begin{teo}\label{teolog}
Let $p \in \PP^{log} (\Omega)$ with $1<p_-\leq p_+ < \infty$. Then for each $u \in W^{1,p(x)}(\Omega)$ there exists a $p(x)-$quasicontinuous function $v \in W^{1,p(x)}(\Omega)$ such that $u=v$ almost everywhere in $\Omega$.
\end{teo}

\begin{remark}
It is easy to see that two $p(x)-$quasi continuous representatives of a given function $u\in W^{1,p(x)}(\Omega)$ can only differ in a set of zero $p(x)-$capacity. Therefore, the unique $p(x)-$quasi continuous representative (defined $p(x)-$q.e.) of $u\in W^{1,p(x)}(\Omega)$ will be denoted by $\tilde u$.
\end{remark}

The proof of the next proposition can be found in \cite[Section 11.1.11]{Diening}.

\begin{prop} \label{propctp}
Let $v_{j}\rightarrow v$ in $W^{1,p(x)}_{0}(D)$. Then, there is a subsequence $\{v_{j_{k}}\}_{k \in \N}$ such that $\tilde{v}_{j_{k}}\rightarrow \tilde{v}$ $p(x)-$q.e.
\end{prop}

Now we need a characterization of the space $W^{1,p(x)}_{0}(\Omega)$ as the restriction of quasi continuous functions that vanishes quasi everywhere on $\R^N\setminus\Omega$. This theorem is esentialy contained in \cite[Corollary 11.2.5, Theorem 11.2.5]{Diening}. We include here the proof since a minor modification of the above mentioned result is needed and for the reader's convenience.

\begin{teo}[Characterization Theorem]\label{teocaracterizacion}
Let $D\subset {\R}^{N}$ be an open set, $\Omega\subset D$ an open subset and $p \in \mathcal P^{log}(\Omega)$. Then,
$$
u\in W^{1,p(x)}_{0}(\Omega) \Leftrightarrow u\in W^{1,p(x)}_{0}(D)\text{ and } \tilde{u}=0\ p(x)-q.e.\text{ in } D \setminus \Omega.
$$
\end{teo}

\begin{proof} 

Let $u\in W^{1,p(x)}_0(\Omega)$, then, it is immediate that $u\in W^{1,p(x)}_0(D)$.

Now, let $\{\varphi_{n}\}_{n \in \N}\subset C^{\infty}_{c}(\Omega)$ such that $\varphi_{n}\to u$ in $W^{1,p(x)}_{0}(\Omega)$ (and therefore in $W^{1,p(x)}_{0}(D)$). 

Let $\{\varphi_{{n_{j}}}\}_{j \in \N}\subset \{\varphi_{n}\}_{n \in \N}$ be a subsequence such that $\varphi_{n_j}\to \tilde{u}$ $p(x)-$q.e. Then, since $\varphi_{{n_{j}}}=0$ in $D \setminus \Omega$, we have that $\tilde{u}=0$ $p(x)-$q.e. in $D \setminus \Omega$.

To see the converse, let us assume that $D={\R}^{N}$ (or else, we extend by zero). Since $u=u^{+}-u^{-}$, we can assume that $u\geq 0$. Moreover, since $\min\{u,n\}\in W^{1,p(x)}({\R}^{N})$ converges to $u$ in $W^{1,p(x)}({\R}^{N})$, we can assume that $u$ is bounded.
Finally, let us consider $\xi \in C^{\infty}_{c}(B(0,2))$ such that $0\leq \xi \leq 1$ and $\xi\equiv 1$ in $B(0,1)$. Setting $\xi_{n}(x)=\xi(\frac{x}{n})$, we have that $\xi_{n} u$ converges to $u$ in $W^{1,p(x)}({\R}^{N})$. Therefore we can assume that $u(x)=0$ for every $x\in (B(0,R))^{c}$ with $R$ large enough.

Therefore, we need to prove the converse for bounded, compactly supported and nonnegative functions $u\in W^{1,p(x)}_0(\R^N)$ such that $\tilde u = 0$ $p(x)-$q.e. in $\Omega^c$.

Since $\tilde u$ is $p(x)-$quasi continuous, there is a decreasing sequence of open sets $\{W_{n}\}_{n \in \N}$ such that $\cp_{p(x)}(W_{n},D)\to 0$ and $\tilde{u}|_{{\R}^{N} \setminus W_{n}}$ is continuous.

We can assume that $W_{n}$ contains the set of null capacity of ${\R}^{N} \setminus \Omega$ where $\tilde{u}\neq 0$. Therefore, $\tilde{u}= 0$ in $(\Omega\cup W_{n})^{c}=\Omega^{c}\cap W_{n}^{c}$.

Given $\delta>0$, set $V_{n}=\{x \colon \tilde{u}(x)<\delta\}\cup W_{n}$. Since $\tilde u$ is continuous in $\R^N\setminus W_n$, $V_{n}$ is an open set. Therefore, $V_{n}^{c}$ is a closed set. It is also bounded since $V_{n}^{c}\subset B(0,R)$. Then, $ V_{n}^{c}$ is compact.

Let $u_{W_{n}}$ be the capacitary potential of $W_{n}$, then $(u-\delta)^{+} (1-u_{W_{n}})=0$ a.e. in $\Omega \setminus V_{n}^{c}$.

Consider now a regularizing sequence $\{\phi_{j}\}_{j \in \N}$. Therefore, for $j$ sufficiently large we have that
$$
\phi_{j}*\left[(u-\delta)^{+} (1-u_{W_{n}}) \right]\in C_{}^{\infty}(\Omega).
$$
Observe that 
$$
\rho_{p(x)}(\nabla u_{W_{n}})=\cp_{p(x)}(W_{n},D)\rightarrow 0.
$$
By Proposition \ref{propdesigualdades}, we can conclude that $\|\nabla u_{W_n}\|_{p(x)}\to 0$ and, by Poincar\'e's inequality, $\|u_{W_n}\|_{1,p(x)}\to 0$. Therefore, $1-u_{W_n} \to 1$ in $W^{1,p(x)}(D)$ when $n\to \infty$. 

Obviously, $(u-\delta)^{+}\to u^{+}=u$ in $W^{1,p(x)}(D)$ when $\delta\to 0$ and observe that
\begin{align*}
\left\|(u-\delta)^{+}(1-u_{W_{n}})-u \right\|_{1,p(x)} \leq &\left\|1-u_{W_{n}}\right\|_{1,p(x)} \left\|(u-\delta)^{+}-u\right\|_{1,p(x)}\\
&+\left\|u \right\|_{1,p(x)} \left\|u_{W_{n}}\right\|_{1,p(x)}.
\end{align*}

Finally, taking the limit when when $j\to \infty$, $n\to \infty$ and $\delta\to 0$, we have that
$$
\phi_{j}*\left[(u-\delta)^{+} (1-u_{W_{n}}) \right]\to u,
$$
which completes the proof. 
\end{proof}

We end this subsection with a lemma that will be much helpful in the sequel.
\begin{lema} \label{lemainf}
Let $v \in W^{1,p(x)}_{0}(\mathbb{R}^{N})$ and $w \in W_{0}^{1,p(x)}(D)$ such that $|v|\leq w$ a.e. in $D$. Then, $v \in W_{0}^{1,p(x)}(D)$. 
\end{lema}

\begin{proof}
It is enough to see that $v^{+} \in W_{0}^{1,p(x)}(D)$ (for $v^{-}$ we prodece similarly and haven shown this result for $v^{+}$ and $v^{-}$, we can state that is valid for $v=v^{+}-v^{-}$).

Since $w\geq 0$, by density we can consider $\{w_{n}\}_{n \in \N}\subset C^{\infty}_{c}(D)^{+}$ such that $\{w_{n}\}_{n \in \N}$ converges to $w$ in $W^{1,p(x)}(D)$.

Therefore, $\inf\{w_{n}, v^{+}\}$, which has compact support in $D$ (for each $w_{n}$ has so) converges to $\inf\{w, v^{+}\}$ which coincides with $v^{+}$ since $|v|\leq w$ a.e. in $D$.

Then, taking an adequate regularizing sequence, we obtain a sequence of $C^{\infty}_{c}(D)$ convergent to $v^{+}$, which completes the proof.
\end{proof}

\section{The Dirichlet problem for the $p(x)-$laplacian.}

We define the $p(x)-$laplacian as 
$$
\Delta_{p(x)}u:=\div(\left|\nabla u \right|^{p(x)-2}\nabla u).
$$
Observe that when $p(x)=2$ this operator agrees with the classical Laplace operator, and when $p(x)=p$ is constant is the well-known $p-$laplacian.

The Dirichlet problem for the $p(x)-$laplacian consists of finding $u$ $\in W^{1,p(x)}_{0}(\Omega)$ such that 
\begin{equation}\label{eq.dirichlet}
\left\{
\begin{array}{rl}
-\Delta_{p(x)}u=f & \text{en } \Omega,\\
u=0 & \text{en } \partial\Omega,
\end{array} \right.
\end{equation}
where $f\in L^{p'(x)}(\Omega)$ or, more generally, $f\in W^{-1,p'(x)}(\Omega)$.

In its weak formulation, this problem consists of finding $u$ $\in W^{1,p(x)}_{0}(\Omega)$ such that 
$$
\int_{\Omega}\left|\nabla u \right|^{p(x)-2}\nabla u \nabla v\, dx=\langle f, v\rangle \text{ for every } v \in W^{1,p(x)}_{0}(\Omega).
$$

Setting 
$$
I(v):= \int_{\Omega}\frac{1}{p(x)}\left|\nabla v \right|^{p(x)}dx - \langle f, v\rangle,
$$
the problem can be reformulated as finding $u \in W^{1,p(x)}_{0}(\Omega)$ such that $$
I(u)=\min\{I(v) \colon v \in W^{1,p(x)}_{0}(\Omega)\}.
$$

By standard methods, we obtain the following result

\begin{teo}
Assume $p_->1$. Then there exists a unique minimizer of $I(v)$ in $W^{1,p(x)}_0(\Omega)$ and a unique weak solution of \eqref{eq.dirichlet} $u\in W^{1,p(x)}_{0}(\Omega)$.
\end{teo}

\begin{proof}
The proof is standard and uses the direct method of the calculus of variations. We omit the details.
\end{proof}

\begin{remark}
The unique weak solution of \eqref{eq.dirichlet} will be denoted by $u^{f}_{\Omega}$.
\end{remark}

\begin{prop}\label{acotacion.solucion}
Let $f\in W^{-1,p'(x)}(\Omega)$ and let $\A>0$ be such that $\|f\|_{-1,p'(x)}\le \A$. Then, there exists a constant $C$ depending only on $\A$, $p_-$ and $p_+$ such that
$$
\|\nabla u_\Omega^f\|_{p(x)}\le C.
$$
\end{prop}

\begin{proof}
Let us assume that $\|\nabla u_\Omega^f\|_{p(x)}>1$ (otherwise, the result is clear). By Proposition \ref{propdesigualdades},  
$$
\int_\Omega |\nabla u_\Omega^f|^{p(x)}=\langle f,u_\Omega^f \rangle \leq \|f\|_{-1,p'(x)}\|u_\Omega^f\|_{p(x)}\leq \|f\|_{-1,p'(x)} (\rho_{p(x)}(u_\Omega^f))^{\frac{1}{p_-}}.
$$
Therefore,
$$
\int_\Omega |\nabla u_\Omega^f|^{p(x)}\leq \|f\|_{-1,p'(x)}^{\frac{p_-}{p_{-}-1}},
$$
which completes the proof.
\end{proof}

In what follows, the monontonicity of the $p(x)-$laplacian is crucial. This fact is a consequence of the following well-known lemma that is proved in \cite[p.p. 210]{Simon}.
\begin{lema}\label{lemadesigualdad}
There is a constant $c_{1}>0$ such that for every $a,b\in {\R}^{N}$,
$$
(|b|^{p-2}b-|a|^{p-2}a)\cdot (b-a)  \geq \begin{cases}
c_{1}|b-a|^{p} & \text{if } p \geq 2,\\
c_{1}\frac{|b-a|^{2}}{(|b|+|a|)^{2-p}} & \text{if } p \leq 2.
\end{cases}
$$
\end{lema}

\begin{remark}\label{deltap}
Observe that if $u\in W^{1,p(x)}_0(\Omega)$, then $-\Delta_{p(x)}u\in W^{-1,p'(x)}(\Omega)$. In fact,
$$
\langle -\Delta_{p(x)}u, v\rangle = \int_{\Omega} |\nabla u|^{p(x)-2}\nabla u\nabla v\, dx.
$$
\end{remark}

\begin{defi}
Let $f\in W^{-1,p'(x)}(\Omega)$. We say that $f\ge 0$ if $\langle f, v\rangle\ge 0$ for every $v\in W^{1,p(x)}_0(\Omega)$ such that $v\ge 0$.

Let $f, g\in W^{-1,p'(x)}(\Omega)$. We say that $g\le f$ if $f-g\ge 0$.
\end{defi}

We now prove the comparison principle for \eqref{eq.dirichlet}
\begin{lema}[Comparison Principle] \label{ppiomax}
Let $u,v \in W_{0}^{1,p(x)}(D)$ be such that
$$\begin{cases}
-\Delta_{p(x)}u\leq -\Delta_{p(x)}v & \text{in } D,\\
u\leq v & \text{on } \partial D.
\end{cases}
$$
Then, $u\leq v$ en $D$.
\end{lema}

\begin{proof}
Let us call $g:=-\Delta_{p(x)}u$ and $f:=-\Delta_{p(x)}v$. Then, by Remark \ref{deltap}, we obtain that, given $\varphi \in W_{0}^{1,p(x)}(D)$,
$$
\int_{D}(|\nabla u|^{p(x)-2}\nabla u-|\nabla v|^{p(x)-2}\nabla v) \nabla \varphi(x)\, dx=\langle g-f, \varphi \rangle.
$$
In particular, taking $\varphi=(u-v)^{+} \in W_{0}^{1,p(x)}(D)$, since $g \leq f$ we have that
$$
\int_{D}(|\nabla u|^{p(x)-2}\nabla u-|\nabla v|^{p(x)-2}\nabla v) \nabla (u-v)^{+}\, dx
=\langle g-f, (u-v)^{+}\rangle \leq 0. 
$$
Taking into account that $\nabla (u-v)^{+}=(\nabla u-\nabla v)\chi_{\{u>v\}}$, we conclude that
$$
\int_{\{u>v\}}(|\nabla u|^{p(x)-2}\nabla u-|\nabla v|^{p(x)-2}\nabla v)(\nabla u-\nabla v)\, dx\leq 0.
$$
Now, let us define $\Omega_1':=\{x\in D \colon p(x)\ge 2\}$ and $\Omega_1'':=\{x\in D\colon p(x)< 2\}$. Therefore, $D = \Omega_1' \cup \Omega_1''$ (disjoint union).

Now, by Lemma \ref{lemadesigualdad}, there is a constant $c>0$ such that
\begin{align*}
&\int_{\{u\geq v\}}(|\nabla u|^{p(x)-2}\nabla u-|\nabla v|^{p(x)-2}\nabla v) (\nabla u-\nabla v)\, dx \\
&\geq c \int_{\{u\geq v\}\cap \Omega_{1}'}|\nabla u-\nabla v|^{p(x)}\, dx + c  \int_{\{u\geq v\}\cap \Omega_{1}''}\frac{|\nabla u-\nabla v|^{2}}{(|\nabla u|+|\nabla v|)^{2-p(x)}}\, dx.
\end{align*}
Therefore, since $\nabla (u-v)^{+}=(\nabla u-\nabla v)\chi_{u>v}$, we conclude that
$$
0 \geq \int_{\Omega_{1}'} |\nabla (u-v)^{+}|^{p(x)}\, dx, + \int_{\Omega_{1}''}\frac{|\nabla (u-v)^{+}|^{2}}{(|\nabla u|+|\nabla v|)^{2-p(x)}}\, dx.
$$
Then, $\nabla (u-v)^{+}=0$ in $D$. So $(u-v)^{+}$ is constant in $D$. Since $(u-v)^{+} \in W_{0}^{1,p(x)}(D)$, we have that $(u-v)^{+}=0$. Therefore $u-v\leq 0$, which completes the proof.
\end{proof}

\begin{corol}[Weak maximum principle]
Let $f\in W^{-1,p'(x)}(\Omega)$ be such that $f\ge 0$. Then $u_\Omega^f\ge 0$.
\end{corol}

\begin{proof}
Just apply Lemma \ref{ppiomax} with $u=0$ and $v=u_\Omega^f$.
\end{proof}

The following proposition gives the monotonicity property of the solution with respect to the domain. The proof follows the ideas of \cite[Theorem 3.2.5.]{Henrot} where the linear case $p(x)=2$ is treated. Nevertheless, since the $p(x)-$laplacian is nonlinear, the monotonicity property of this operator comes into play replacing linearity in the argument.

\begin{prop}[Property of monotonicity with respect to the domain.]\label{propmonotonia}
Let $\Omega_{1}\subset \Omega_{2}$ and $f\in W^{-1,p'(x)}(\Omega_2)$ be such that $f\geq 0$. Then, $u_{\Omega_1}^{f}\leq u_{\Omega_2}^{f}$.
\end{prop}

\begin{proof}
We will denote $u_{1}=u_{\Omega_1}^{f}$ and $u_{2}=u_{\Omega_2}^{f}$.

Given $v \in W_{0}^{1,p(x)}(\Omega_{1})\subset W^{1,p(x)}_0(\Omega_2)$,
\begin{equation}\label{primeraintegral}
\int_{\Omega_{i}}|\nabla u_{i}|^{p(x)-2}\nabla u_{i} \nabla v\, dx=\langle f, v\rangle,\quad  i=1,2. 
\end{equation}
Therefore,
\begin{equation}\label{segundaintegral} 
\int_{\Omega_{1}}(|\nabla u_{1}|^{p(x)-2}\nabla u_{1}-|\nabla u_{2}|^{p(x)-2}\nabla u_{2}) \nabla v\, dx=0,
\end{equation}
for every $v\in W^{1,p(x)}_0(\Omega_1)$.

Since $f\geq 0$, we have that $u_{2}\geq 0$. Then, $(u_{1}-u_{2})^{+}\leq u_{1}^{+} \in W_{0}^{1,p(x)}(\Omega_{1})$ and hence, by Lemma \ref{lemainf}, $(u_{1}-u_{2})^{+} \in W_{0}^{1,p(x)}(\Omega_{1})$. Therefore
$$
\int_{\Omega_{1}}(|\nabla u_{1}|^{p(x)-2}\nabla u_{1}-|\nabla u_{2}|^{p(x)-2}\nabla u_{2}) \nabla (u_{1}-u_{2})^{+}\, dx=0.
$$
Now, let us define $\Omega_1':=\{x\in \Omega_1\colon p(x)\ge 2\}$ and $\Omega_1'':=\{x\in \Omega_1\colon p(x)< 2\}$. Therefore, $\Omega_1 = \Omega_1' \cup \Omega_1''$ (disjoint union).

Now, by Lemma \ref{lemadesigualdad}, there is a constant $c>0$ such that
\begin{align*}
0&=\int_{\{u_{1}\geq u_{2}\}\cap \Omega_{1}}(|\nabla u_{1}|^{p(x)-2}\nabla u_{1}-|\nabla u_{2}|^{p(x)-2}\nabla u_{2}) (\nabla u_{1}-\nabla u_{2})\, dx \\
&\geq c \int_{\{u_{1}\geq u_{2}\}\cap \Omega_{1}'}|\nabla u-\nabla v|^{p(x)}\, dx + c  \int_{\{u_{1}\geq u_{2}\}\cap \Omega_{1}''}\frac{|\nabla u-\nabla v|^{2}}{(|\nabla u|+|\nabla v|)^{2-p(x)}}\, dx.
\end{align*}
Therefore, since $\nabla (u-v)^{+}=(\nabla u-\nabla v)\chi_{u>v}$, we conclude that
$$
0 \geq \int_{\Omega_{1}'} |\nabla (u-v)^{+}|^{p(x)}\, dx, + \int_{\Omega_{1}''}\frac{|\nabla (u-v)^{+}|^{2}}{(|\nabla u|+|\nabla v|)^{2-p(x)}}\, dx.
$$
Then, $\nabla (u_{1}-u_{2})^{+}=0$ in $\Omega_{1}$. Hence, $(u_{1}-u_{2})^{+}$ is constant in $\Omega_{1}$. Since $(u_{1}-u_{2})^{+} \in W_{0}^{1,p(x)}(\Omega_{1})$, we have that $(u_{1}-u_{2})^{+}=0$. Therefore $u_{1}-u_{2}\leq 0$, which completes the proof.
\end{proof}

We now end this section with an stability result for solutions of the Dirichlet problem
\begin{teo}\label{estabilidadDirichlet}
Let $D\subset \R^N$ be open, and let $f_i\in W^{-1,p'(x)}(D)$, $i=1,2$. There exists a constant $C>0$ depending only on $p_-$, $p_+$ and $\max\{\|f_i\|_{-1,p'(x)}\}$ such that, if $\Omega\subset D$, 
$$
\int_{D} |\nabla u_\Omega^{f_1} - \nabla u_\Omega^{f_2}|^{p(x)}\, dx \le C (\|f_1-f_2\|_{-1,p'(x)} + \|f_1-f_2\|_{-1,p'(x)}^\beta),
$$
where the constant $\beta>0$ depends only on $p_-$ and $p_+$.
\end{teo}

Theorem \ref{estabilidadDirichlet} immediately implies the following Corollary. 
\begin{corol}\label{continuidad.dato}
Let $f_n, f\in W^{-1,p'(x)}(\Omega)$ be such that $\|f_n-f\|_{-1,p'(x)}\to 0$. Then 
$$
\|\nabla u_\Omega^{f_n} - \nabla u_\Omega^{f}\|_{p(x)}\to 0.
$$
\end{corol}

Now we proceed with the proof of the Theorem.

\begin{proof}[Proof of Theorem \ref{estabilidadDirichlet}]
Let us denote $u_i = u_{\Omega}^{f_i}$. 

Given $\varphi \in W_{0}^{1,p(x)}(\Omega)$, we have that
\begin{equation}\label{uno}
\int_{\Omega}|\nabla u_i|^{p(x)-2}\nabla u_i \nabla \varphi \, dx=\langle f_i, \varphi \rangle,\quad i=1,2.
\end{equation}
In particular, considering $\varphi=u_1-u_2\in W_{0}^{1,p(x)}(\Omega)$ and subtracting, we obtain
\begin{align*}
\int_{\Omega}(|\nabla u_1|^{p(x)-2}\nabla u_1 - |\nabla u_2|^{p(x)-2}\nabla u_2)& (\nabla u_1 - \nabla u_2)\, dx \\
&= \langle f_{1}-f_{2}, u_1-u_2\rangle\\
& \leq  \|f_{1}-f_{2}\|_{-1,p'(x)} \|\nabla u_1 - \nabla u_2\|_{p(x)}\\
&\leq \|f_{1}-f_{2}\|_{-1,p'(x)} (\|\nabla u_1\|_{p(x)} + \|\nabla u_2\|_{p(x)})\\
&\le C\|f_{1}-f_{2}\|_{-1,p'(x)}
\end{align*}
where we have used Proposition \ref{acotacion.solucion} in the last inequality.

On the other hand, naming $\Omega_{1}=\Omega \cap \{p(x)\geq 2\}$ and $\Omega_{2}=\Omega \cap \{p(x)< 2\}$, we have that
\begin{align*}
\int_{\Omega}&(|\nabla u_1|^{p(x)-2}\nabla u_1-|\nabla u_2|^{p(x)-2}\nabla u_2) (\nabla u_1 - \nabla u_2)\, dx\\
&=\sum_{i=1}^2 \int_{\Omega_{i}}(|\nabla u_1|^{p(x)-2}\nabla u_1 - |\nabla u_2|^{p(x)-2}\nabla u_2) (\nabla u_1 - \nabla u_2)\, dx.
\end{align*}
Let us study each of these integrals. 
By Lemma \ref{lemadesigualdad},
$$
\int_{\Omega_{1}}(|\nabla u_1|^{p(x)-2}\nabla u_1 - |\nabla u_2|^{p(x)-2}\nabla u_2) (\nabla u_1-\nabla u_2)\, dx 
\geq c \int_{\Omega_{1}} |\nabla (u_1-u_2)|^{p(x)}\, dx.
$$
Let us now analyze the integral over $\Omega_{2}$.
\begin{align*}
\int_{\Omega_{2}} |\nabla (u_1-u_2)|^{p(x)}dx &=  \int_{\Omega_{2}} (|\nabla u_1|+|\nabla u_2|)^{\frac{(2-p(x))p(x)}{2}} \left(\frac{|\nabla (u_1-u_2)|}{(|\nabla u_1|+|\nabla u_2|)^{\frac{2-p(x)}{2}}}\right)^{p(x)}\, dx\\
&\leq 2 \|(|\nabla u_1|+|\nabla u_2|)^{\frac{(2-p(x))p(x)}{2}}\|_{\frac{2}{2-p(x)}} \left\| \left(\frac{|\nabla (u_1-u_2)|}{(|\nabla u_1|+|\nabla u_2|)^{\frac{2-p(x)}{2}}}\right)^{p(x)}\right\|_{\frac{2}{p(x)}}\\
&\le 2 \left(\int_{\Omega_2}(|\nabla u_1|+|\nabla u_2|)^{p(x)}dx \right)^{\alpha}
\left(\int_{\Omega_2}\frac{|\nabla (u_1-u_2)|^2}{(|\nabla u_1|+|\nabla u_2|)^{2-p(x)}}\, dx\right)^{\beta}.
\end{align*}
for some constants $\alpha$ and $\beta$ depending only on $p_-$ and $p_+$. Let us observe that for the first inequality we took into account H\"older's inequality and for the second one, Observation \ref{minmax}.

Let us now find a bound for the first factor. In fact, by Proposition \ref{acotacion.solucion}.
\begin{align*}
\int_{\Omega_{2}}(|\nabla u_1|+|\nabla u_2|)^{p(x)}\, dx & \leq 2^{p_+ - 1} \int_{\Omega_{2}}(|\nabla u_1|^{p(x)}+|\nabla u_2|^{p(x)})\, dx\le C.
\end{align*}

Observe that, by Lemma \ref{lemadesigualdad}, we are able to find a bound for the second factor.
$$
\int_{\Omega_2}\frac{|\nabla (u_1-u_2)|^2}{(|\nabla u_1|+|\nabla u_2|)^{2-p(x)}}\, dx \leq C \int_{\Omega_2} (|\nabla u_1|^{p(x)-2}\nabla u_1-|\nabla u_2|^{p(x)-2}\nabla u_2)(\nabla u_1-\nabla u_2)\, dx.
$$
Then,
\begin{align*}
\int_{\Omega_2} |\nabla (u_1-u_2)|^{p(x)}\, dx &\leq
C \left( \int_{\Omega_{2}} (|\nabla u_1|^{p(x)-2}\nabla u_2 - |\nabla u_2|^{p(x)-2}\nabla u_2) (\nabla u_1 - \nabla u_2)\, dx\right)^{\beta}\\
&\le C \left( \int_{\Omega} (|\nabla u_1|^{p(x)-2}\nabla u_2 - |\nabla u_2|^{p(x)-2}\nabla u_2) (\nabla u_1 - \nabla u_2)\, dx\right)^{\beta}\\
&\le C \|f_1-f_2\|_{-1,p'(x)}^\beta.
\end{align*}
So we can conclude that
$$
\int_{\Omega} |\nabla (u_1-u_2)|^{p(x)}\, dx \le C (\|f_1-f_2\|_{-1,p'(x)} + \|f_1-f_2\|_{-1,p'(x)}^\beta).
$$
This finishes the proof.
\end{proof}

\section{Continuity of the Dirichlet problem with respect to perturbations on the domain.}

In this section we investigate the dependence of the solutions of the Dirichlet problem $u^f_\Omega$ with respect to perturbations on the domain. We will analyze a rather general problem considering a sequence of uniformly bounded domains $\Omega_n$ converging to a limiting domain $\Omega$ in the Haussdorf complementary topology. Then we study whether $u^f_{\Omega_n}$ converges to $u^f_\Omega$ or not.

We begin this section by defining a notion of convergence of domains that will be essential for our next results. 

\begin{defi}[Hausdorff complementary topology.] Let $D\subset \R^N$ be compact. Given $K_1, K_2\subset D$ compact sets, we define de Hausdorff distance $d_H$ as
$$
d_H(K_1, K_2) := \max\left\{\sup_{x\in K_1} \inf_{y\in K_2} \|x-y\|, \sup_{x\in K_2} \inf_{y\in K_1} \|x-y\|\right\}.
$$

Now, let $\Omega_1,\Omega_2\subset D$ be open sets, we define the Hausdorff complementary distance $d^H$ as
$$
d^H(\Omega_1, \Omega_2) := d_H(D \setminus \Omega_1, D \setminus \Omega_2).
$$

Finally, we say that $\{\Omega_n\}_{n\in \N}$ converges to $\Omega$ in the sense of the Hausdorff complementary topology, denoted by $\Omega_n \stackrel{H}{\to}\Omega$, if $d^H(\Omega_n, \Omega)\to 0$.
\end{defi}
For an study and properties of this topology of open sets, we refer to the book \cite{Henrot}.

We now present the one property that will be essential  for our purposes. 

\begin{prop}
Let $K\subset \Omega$ be a compact set. If $\Omega_n\stackrel{H}{\to}\Omega$, then $K\subset \Omega_n$ for every $n$ large enough.
\end{prop}

\begin{proof}
The proof is immediate from the definition. See \cite{Henrot}.
\end{proof}

Now we state a couple of corollaries of Proposition \ref{acotacion.solucion} that will be most useful.
\begin{corol}\label{corocontinuidad}
Let $D\subset \mathbb{R}^{N}$ be an open bounded set and let  $\Omega_n \subset D$ be a sequence of open domains. Let $p \in \PP^{log}(\Omega)$ such that $p_->1$. Then, $\{u_{\Omega_n}^{f}\}_{n\in\N}$ is bounded in ${ W^{1,p(x)}_{0}(D)}$.
\end{corol}

\begin{corol}\label{lemacontinuidad}
Under the same assumptions as in Corollary \ref{corocontinuidad}, we have that the sequence $\{|\nabla u_{\Omega_n}^f|^{p(x)-2} \nabla u_{\Omega_n}^f\}_{n \in \N}$ is bounded in $L^{p'(x)}(D)$.
\end{corol}

We now extend to variable exponent spaces Proposition 3.7 in \cite{Bucur}.

\begin{teo}\label{teo1}
Let us denote $u_{n} = u_{\Omega_n}^f$. Assume that $u_{n}\rightharpoonup u^{*}$ weakly in $W^{1,p(x)}_{0}(D)$. Let $\Omega \subset D$ be such that for every compact subset $K \subset \Omega$, there is an integer $n_{0}$ such that $K \subset \Omega_{n}$ for every $n\geq n_{0}$. Then, there holds that 
$$
-\Delta_{p(x)}u^{*}=f \text{ in } \Omega.
$$
\end{teo}

\begin{remark}\label{loquefalta}
Observe that in order to conclude that $u^* = u_{\Omega}^f$ it remains to see that $u^*\in W^{1,p(x)}_0(\Omega)$. 
\end{remark}

\begin{proof}
As $p\in \PP^{log}(D)$, we need to verify that, given $\varphi \in C^{\infty}_{c}(\Omega)$, the following equality is valid:
$$
\int_{\Omega}|\nabla u^{*}|^{p(x)-2}\nabla u^{*}\nabla \varphi \, dx=\langle f,\varphi \rangle.
$$

Let $\varphi \in C^{\infty}_{c}(\Omega)$. Since $\sop(\varphi)\subset \Omega$ is compact, there is an integer $n_{0}$ such that $\sop(\varphi)\subset \Omega_{n}$ for every $n\geq n_{0}$. Therefore, $\varphi \in C^{\infty}_{c}(\Omega_{n})$ for every $n\geq n_{0}$.

Set $K=\sop(\varphi)$ and $K^{\ve}=\{x \in \mathbb{R}^{N} \colon d(x,K)<\ve \}$ with $\ve$ sufficiently small to make sure that $K^{\ve}\subset \subset\Omega_{n}\cap \Omega$ for every $n\geq n_{1}$.

We will, from now on, work with $n\geq \max\{n_{0},n_{1}\}$.

Let $\eta \in C^{\infty}_{c}(\Omega)$ such that $\eta=1$ in $K^{\frac{\ve}{2}}$, $\eta=0$ in $(K^{\ve})^{c}$ and $0\leq \eta \leq 1$.

Consider $\phi_{n}=\eta (u_{n}-u^{*})$ and since $\phi_n\in W^{1,p(x)}_0(\Omega_n)$ we have
$$
\int_{D}|\nabla u_{n}|^{p(x)-2}\nabla u_{n}\nabla \phi_{n} \, dx=\int_{\Omega_{n}}|\nabla u_{n}|^{p(x)-2}\nabla u_{n}\nabla \phi_{n} \, dx= \langle f, \phi_{n} \rangle.
$$
Standard computations now give us
$$
\int_{D}|\nabla u_{n}|^{p(x)-2}\nabla u_{n} \eta \nabla (u_{n}-u^{*})\, dx \leq \langle f, \phi_n\rangle - \int_{D}|\nabla u_{n}|^{p(x)-2}\nabla u_{n} \nabla \eta (u_{n}-u^{*})\, dx.
$$
Since $u_{n}\rightharpoonup u^{*}$ en $W^{1,p(x)}_{0}(D)$, $\phi_n\rightharpoonup 0$ in $W^{1,p(x)}_0(\Omega)$ and so $\langle f, \phi_n\rangle\to 0$.

On the one hand, by the compactness of the embedding $W^{1,p(x)}_0(D)\subset L^{p(x)}(D)$, we have that $u_{n}\rightarrow u^{*}$ in $L^{p(x)}(D)$, and so, by H\"older's inequality,
$$
\int_{D}|\nabla u_{n}|^{p(x)-2}\nabla u_{n} \nabla \eta (u_{n}-u^{*})\, dx\le 2\|\nabla \eta\|_\infty \| |\nabla u_n|^{p(x)-2}\nabla u_n\|_{p'(x)} \|u_n-u\|_{p(x)} \rightarrow 0,
$$
by Corollary \ref{lemacontinuidad}. Then we can conclude that
$$
\limsup \int_{D}|\nabla u_{n}|^{p(x)-2}\nabla u_{n} \eta \nabla (u_{n}-u^{*})\, dx\leq 0.
$$
Since $\eta=0$ en $(K^{\ve})^{c}$,
\begin{equation}\label{primera}
\limsup \int_{K^{\ve}}|\nabla u_{n}|^{p(x)-2}\nabla u_{n} \eta \nabla (u_{n}-u^{*})\, dx\leq 0\
\end{equation}

On the other hand, since $\nabla u_{n}\rightharpoonup \nabla u^{*}$ en $L^{p'(x)}(K^{\ve})$,
\begin{equation}\label{segunda}
\int_{K^{\ve}}|\nabla u^{*}|^{p(x)-2}\nabla u^{*} \eta \nabla (u_{n}-u^{*})\, dx\rightarrow 0\
\end{equation}
By \eqref{primera} and \eqref{segunda} we have that
$$
\limsup \int_{K^{\ve}}(|\nabla u_{n}|^{p(x)-2}\nabla u_{n}-|\nabla u^{*}|^{p(x)-2}\nabla u^{*})\eta \nabla (u_{n}-u^{*})\, dx\leq 0.
$$
Since $K^{\frac{\ve}{2}}\subset K^{\ve}$, by Lemma \ref{lemadesigualdad}, we can conclude that
$$
\lim \int_{K^{\frac{\ve}{2}}}(|\nabla u_{n}|^{p(x)-2}\nabla u_{n}-|\nabla u^{*}|^{p(x)-2}\nabla u^{*})\nabla (u_{n}-u^{*})\, dx=0.
$$
Again, by Lemma \ref{lemadesigualdad}, it follows that $(|\nabla u_{n}|^{p(x)-2}\nabla u_{n} - |\nabla u^{*}|^{p(x)-2}\nabla u^{*})\nabla (u_{n}-u^{*})\to 0$ in $L^{1}(K^{\frac{\ve}{2}})$ and therefore a.e. in $K^{\frac{\ve}{2}}$.

From these facts, it easily follows that 
\begin{equation}\label{conv.puntual}
\nabla u_{n}\rightarrow \nabla u^{*} \text{ a.e. in } K^{\frac{\ve}{2}}.
\end{equation} 

Finally, by Corollary \ref{lemacontinuidad}, there exists $\xi\in L^{p(x)}(K^{\frac{\ve}{2}})^N$ such that $|\nabla u_{n}|^{p(x)-2}\nabla u_{n}\rightharpoonup \xi$ in $L^{p'(x)}(K^{\frac{\ve}{2}})$.

From \eqref{conv.puntual}, we can conclude that $\xi=|\nabla u^{*}|^{p(x)-2}\nabla u^{*}$ in $K^{\frac{\ve}{2}}$ and that
$$
\int_{K^{\frac{\ve}{2}}}|\nabla u_{n}|^{p(x)-2}\nabla u_{n}\nabla \varphi \, dx\rightarrow \int_{K^{\frac{\ve}{2}}}|\nabla u^{*}|^{p(x)-2}\nabla u^{*}\nabla \varphi \, dx.
$$
Since $\sop(\nabla \varphi)\subset K\subset K^{\frac{\ve}{2}}\subset K^{\ve}\subset \Omega_{n}\cap \Omega$,
$$
\int_{\Omega_{n}}|\nabla u_{n}|^{p(x)-2}\nabla u_{n} \nabla \varphi \, dx\rightarrow \int_{\Omega}|\nabla u^{*}|^{p(x)-2}\nabla u^{*}\nabla \varphi \, dx.
$$
This finishes the proof.
\end{proof}

As we mentioned in Remark \ref{loquefalta}, in order to obtain the continuity of solutions with respect to the domain, we need to provide with conditions than ensure $u^{*}\in W^{1,p(x)}_{0}(\Omega)$. This is the content of the next theorem.

\begin{teo}\label{teo2dirichlet}
Let $D\subset {\R}^{N}$ be an open bounded set and let $\Omega_{n}, \Omega \subset D$ be open for every $n$. Let $p\in \PP^{log}(D)$. If $\Omega_n \stackrel{H}{\to}\Omega$ and $\cp_{p(x)}(\Omega_{n} \setminus \Omega,D)\to 0$, then $u_{\Omega_n}^f\rightharpoonup u_{\Omega}^f$ weakly in $W^{1,p(x)}_{0}(D)$.
\end{teo}

\begin{proof}
As before, we denote $u_n = u_{\Omega_n}^f$. By Corollary \ref{corocontinuidad}, $\{u_n\}_{n\in\N}$ is bounded in $W^{1,p(x)}_{0}(D)$, therefore, we can assume that $u_{n}\rightharpoonup u^{*}$ weakly in $W^{1,p(x)}_{0}(D)$.

By Theorem \ref{teo1} and Remark \ref{loquefalta} the proof will be finished if we can prove that $u^{*}\in W^{1,p(x)}_{0}(\Omega)$. By Theorem \ref{teocaracterizacion}, it is enough to prove that  $\tilde{u^{*}}=0$ $p(x)-$q.e. in $\Omega^{c}$. 

Consider $\tilde{\Omega}_{j}=\displaystyle{\cup_{n\geq j}}\Omega_{n}$ and $E=\displaystyle{\cap_{j\geq 1}}\tilde{\Omega}_{j}$.

Since  $u_{n}\rightharpoonup u^{*}$ in $W^{1,p(x)}_{0}(D)$, by Mazur's Lemma (see for instance \cite{E-T}), there is a sequence $v_{j}=\sum^{k_{j}}_{n=j}a_{n_{j}}u_{n}$ such that $a_{n_{j}}\geq 0$, $\sum^{k_{j}}_{n=j}a_{n_{j}}=1$ and $v_{j}\rightarrow u^{*}$ in $W^{1,p(x)}_{0}(D)$.

Since $u_{n}\in W^{1,p(x)}_{0}(\Omega_{n})$, by Theorem \ref{teocaracterizacion}, $\tilde{u}_{n}=0$ $p(x)-$q.e. in $\Omega_{n}^{c}$. Therefore, $\tilde{v}_{j}=\sum^{k_{j}}_{n=j}a_{n_{j}}\tilde{u}_{n}=0$ $p(x)-$q.e. in $\cap^{k_{j}}_{n=j}\Omega_{n}^{c} \supset \tilde{\Omega}_{j}^{c}$ for every $j\geq 1$. Then, $\tilde{v}_{j}=0$ $p(x)-$q.e. in $\tilde{\Omega}_{j}^{c}$ for every $j\geq 1$. As a consequence, $\tilde{v}_{j}=0$ $p(x)-$q.e. $\cup_{j\geq 1}\tilde{\Omega}_{j}^{c}=E^{c}$.

On the other hand, since $v_{j}\rightarrow u^{*}$ in $W^{1,p(x)}_{0}(D)$, by Proposition \ref{propctp} $\tilde{v}_{j_{k}}\rightarrow \tilde{u^{*}}$ $p(x)-$q.e. Then we conclude that $\tilde{u^{*}}=0$ $p(x)-$q.e. in $E^{c}$.

Since $\cp_{p(x)}(\Omega_{n} \setminus \Omega,D)$ goes to zero, passing to a subsequence, if necessary, we can assume that $\cp_{p(x)}(\Omega_{n} \setminus \Omega,D)\leq \frac{1}{2^{n}}$. Therefore,
\begin{align*}
\cp_{p(x)}(\tilde{\Omega}_{j} \setminus \Omega,D)&=\cp_{p(x)}(\cup_{n\geq j}\Omega_{n} \setminus \Omega,D)\\
& \leq \sum_{n\geq j}\cp_{p(x)}(\Omega_{n} \setminus \Omega,D)\\
& \leq \sum_{n\geq j}\frac{1}{2^{n}}=\frac{1}{2^{j-1}}.
\end{align*}
Since $E\subset \tilde{\Omega}_{j}$, we have that $E \setminus \Omega\subset \tilde{\Omega}_{j} \setminus \Omega$ for every $j\geq 1$ and so,
$$
\cp_{p(x)}(E \setminus \Omega,D)\leq \cp_{p(x)}(\tilde{\Omega}_{j} \setminus \Omega,D)\leq \frac{1}{2^{j-1}} \text{ for every } j\geq 1. 
$$
Taking the limit $j\to\infty$, we have that $\cp_{p(x)}(E \setminus \Omega,D)=\cp_{p(x)}(\Omega^{c} \setminus E^{c},D)=0$. So we can conclude that $\tilde{u^{*}}=0$ $p(x)-$q.e. in $\Omega^{c}$, which completes the proof.
\end{proof}

The next result shows that the continuity of the solutions of the Dirichlet problem for the $p(x)-$laplacian with respect to the domain is independent of the second member $f$.

For constant exponents, this result was obtained in \cite[Lemma 4.1]{Bucur}. The proof that we present here, in the non-constant exponent case, follows closely the one in \cite[Theorem 3.2.5]{Henrot} where the linear case $p(x)\equiv 2$ is studied.

\begin{teo}[Independence with respect to the second member] \label{teoindep} Let $\Omega_{n},\Omega\subset D$ be open sets such that $u_{\Omega_{n}}^{1}\rightarrow u_{\Omega}^{1}$ in $L^{p(x)}(D)$. Then $u_{\Omega_{n}}^{f}\rightharpoonup u_{\Omega}^{f}$ in $W_{0}^{1,p(x)}(D)$ for every $f \in W^{-1,p'(x)}(D)$.
\end{teo}

\begin{proof}
Let us assume first that $f \in L^{\infty}(D)$. Therefore, there is a constant $M>0$ such that $-M\leq f \leq M$ a.e. We can also assume that $M>1$.

We will name $u_{n}^{f}=u_{\Omega_{n}}^{f}$ and $u^{f}=u_{\Omega}^{f}$.

Given $k>1$, since $u_{n}^{1}$ is the solution of the equation with $f \equiv 1$,
\begin{align*}
\displaystyle{\int_{\Omega}|\nabla (k u_{n}^{1})|^{p(x)-2}\nabla (k u_{n}^{1}) \nabla \varphi \, dx} &= \displaystyle{\int_{\Omega}k^{p(x)-1}|\nabla u_{n}^{1}|^{p(x)-2}\nabla u_{n}^{1} \nabla \varphi \, dx}\\
                                                                                                & \geq  k^{p_--1}\displaystyle{\int_{\Omega}|\nabla u_{n}^{1}|^{p(x)-2}\nabla u_{n}^{1} \nabla \varphi \, dx}\\
                                                                                                 &= k^{p_--1} \displaystyle{\int_{\Omega}\varphi \, dx}.
\end{align*}

Considering $k=M^{\frac{1}{p_--1}}$, we have that $f \leq M=k^{p_--1}\leq -\Delta_{p(x)}(k u_{n}^{1})$, therefore, $u$ is a supersolution. Since we also have that $0=u_{n}^{f}|_{\partial D}\leq k u_{n}^{1}|_{\partial D}=0$, by Proposition \ref{propmonotonia}, we can conclude that $u_{n}^{f}\leq k u_{n}^{1}$.

On the other hand, since $-\Delta_{p(x)}(-k u_{n}^{1})=\Delta_{p(x)}(k u_{n}^{1})\leq -k^{p_--1}=-M\leq f$, we obtain that $-k u_{n}^{1}\leq u_{n}^{f}$.

Therefore 
\begin{equation}\label{indep}
-k u_{n}^{1}\leq u_{n}^{f}\leq k u_{n}^{1}.                                                   
\end{equation}
By Corollary \ref{corocontinuidad}, $\{u_{n}^{f}\}_{n \in \N}$ is bounded in $W^{1,p(x)}_{0}(D)$. Then, by Alaoglu's Theorem, there is a subsequence, which will remaine denoted $\{u_{n}^{f}\}_{n \in \N}$ such that $u_{n}^{f}\rightharpoonup u^{*}$ en $W_{0}^{1,p(x)}(D)$.

Since, by Rellich-Kondrachov's Theorem, we know that $W_{0}^{1,p(x)}(D)$ is compactedly embbeded in $L^{p(x)}(D)$, we have that $u_{n}^{f}\rightarrow u^{*}$ in $L^{p(x)}(D)$.

Then, taking into account the convergence in $L^{p(x)}(D)$ in \eqref{indep}, we have that
$$
-k u^{1}\leq u^{*}\leq k u^{1}.
$$
Therefore, $|u^{*}|\leq k u^{1}$ and, since $u^{1} \in W_{0}^{1,p(x)}(\Omega)$ we can conclude that $u^{*} \in W_{0}^{1,p(x)}(\Omega)$.

Let us assume now that $f \in W^{-1,p'(x)}(D)$. By density, there is a sequence $\{f_{j}\}_{j \in \N} \subset L^{\infty}(D)$ such that $f_{j}\rightarrow f$ in $W^{-1,p'(x)}(D)$. 

Given $\varphi \in W^{-1,p'(x)}(D)$,
$$
\langle \varphi , u_n^f - u^f \rangle = \langle \varphi,  u_n^f - u_n^{f_j} \rangle + \langle \varphi, u_n^{f_j} - u^{f_j}\rangle + \langle \varphi,  u^{f_j} - u^f \rangle.
$$
Now, by Theorem \ref{estabilidadDirichlet}, given $\ve>0$, there exists $j_0\in \N$ such that
$$
\|\nabla u_n^f - \nabla u_n^{f_j}\|_{p(x)}\le \ve \quad \text{and}\quad \|\nabla u^f - \nabla u^{f_j}\|_{p(x)}\le \ve,
$$
uniformly in $n\in \N$ for every $j\ge j_0$. By the first part of the proof, 
$$
\langle \varphi, u_n^{f_j} - u^{f_{j_0}}\rangle\to 0 \quad \text{as } n\to\infty.
$$
This completes the proof.
\end{proof}

\section{Extension of a result of \v{S}ver\'ak.}

In this section, we apply our results to prove the extension of the theorems of \v{S}ver\'ak discussed in the introduction. Our main result being Theorem \ref{teosverak}.

We begin by establishing some capacity estimate from below for compact connected sets. This was obtained for $p(x)\equiv 2$ by \v{S}ver\'ak in \cite{Sverak}. See the book \cite{Henrot} for a proof. For general constant exponents, this estimate was obtained in \cite{Bucur}. Our extension to variable exponents will rely on Bucur and Trebeschi's result \cite{Bucur}. In fact, we use the following proposition.

\begin{prop}[\cite{Bucur}, Lemma 5.2]\label{lemasverakdos}
Let $p>N-1$ be constant and let $K\subset \R^N$ be compact and connected. Assume that there exists a constant $a>0$ such that $2a<\diam K$. Then, for every $x \in K$ and $a\le r< \frac{\diam K}{2}$, we have the following inequality:
$$
\cp_{p}(K\cap \overline{B(x,r)},B(x,2r))\geq c,
$$
for some constant $c>0$ depending only on $p$ and $a$.
\end{prop}

The next proposition relates the relative capacity of a set for constant exponents with the one with variable exponents.

\begin{prop}\label{cotabis}
Let $p \in \PP^{\log}(D)$. Then,
$$
\cp_{p_{-}}(E,D)\leq C \cp_{p(x)}(E,D)^{\beta},
$$
where$C>0$ depends on $|D|$, $p_+$ and $p_-$ and $\beta>0$ depends on $p_+$ and $p_-$.
\end{prop}

\begin{proof}
Given $\varphi\in W^{1,p(x)}_0(D)$, by H\"older's inequality and Proposition \ref{propdesigualdades}, we obtain
$$
\int_{D}|\nabla \varphi |^{p_-}\, dx \leq C\left(\int_{D}|\nabla \varphi |^{p(x)}\, dx\right)^{\beta}.
$$
So we conclude
$$
\inf_{\varphi \in S_{p(x)}(E,D)}\int_{D}|\nabla \varphi |^{p_-}\, dx \leq  C \left(\inf_{\varphi \in S_{p(x)}(E,D)}\int_{D}|\nabla \varphi |^{p(x)}\, dx\right)^{\beta}.
$$
On the other hand, since $W^{1,p(x)}_{0}(D)\subset W^{1,p_-}_{0}(D)$,
$$
\inf_{\varphi \in S_{p_-}(E,D)}\int_{D}|\nabla \varphi |^{p_-}\, dx \leq  C \left(\inf_{\varphi \in S_{p(x)}(E,D)}\int_{D}|\nabla \varphi |^{p(x)}\, dx\right)^{\beta}.
$$
We can conclude that $\cp_{p_{-}}(E,D)\leq C \cp_{p(x)}(E,D))^{\beta}$.
\end{proof}

From Proposition \ref{lemasverakdos} and Proposition \ref{cotabis} we obtain the following corollary.
\begin{corol}\label{propdefinitivo}
Given $K \subset D \subset \R^N$ compact and connected and $p \in \PP^{log}(B(x,2r))$ such that $p_{-}\ge N-1$. Then, for every $x \in K$ and $a \leq r<\frac{\diam K}{2}$ for some positive constant $a$,
$$
\cp_{p(x)}(K\cap \overline{B(x,r)},B(x,2r))\geq \kappa,
$$
for some constant $\kappa>0$ depending on $|D|$, $\diam D$, $p_+$ and $p_-$.
\end{corol}

\begin{proof}
Just apply Proposition \ref{cotabis} to the sets $K\cap\overline{B(x,r)}$ and $B(x.2r)$, and observe that $2r<\diam K\le \diam D$. Then apply Proposition \ref{lemasverakdos}.
\end{proof}

Now we look for an extension of Theorem \ref{teo2dirichlet} in the sense that instead of requiring some capacity condition on the differences of the approximating domains with the limiting domain, we require a uniform boundary regularity in terms of capacity.

\begin{defi}
We say that $\Omega$ verifies the condition $(p(x), \alpha, r)$ if
$$
\cp_{p(x)}(\Omega^{c}\cap B(x,r),B(x,2r))\geq \alpha,\quad x \in \partial \Omega.
$$ 
Set $\mathcal{O}_{\alpha,r_{0}}(D)=\{\Omega \subset D \text{ open} \colon \Omega \text{ verifies the condition } (p(x), \alpha, r) \text{ for every } 0<r<r_0\}$.
\end{defi}

From now on we will need a result on uniform continuity with respect to $\Omega\in \mathcal{O}_{\alpha, r_0}(D)$ for the solutions of the Dirichlet problem, $u_\Omega^f$ with $f$ sufficiently integrable.

This result for $p(x)\equiv 2$ is classic and can be found, for instance, in \cite[Lemma 3.4.11 and Theorem 3.4.12, p.p. 109]{Henrot}. The key for its proof is to obtain the {\em Wiener conditions}, see \cite{GT}.

The extension for $1<p<N$ constant can be found in the articles \cite{GZ, KM, Mazya}. Consult the book \cite{MZ}, Theorem 4.22. The result for $p(x)$ variable was recently obtained in \cite{Lukkari}.

\begin{lema}[\cite{Lukkari}, Theorem 4.4] \label{lemawiener}
Given $\Omega \in \mathcal{O}_{\alpha,r_{0}}(D)$, $f \in L^{r}(D)$, $r>N$. Then, there are constants $M>0$ and $0<\delta<1$ such that $|u_\Omega^f(x)-u_\Omega^f(y)|\leq M |x-y|^{\delta}$.
\end{lema}

With this result we are able to prove the analogous of Theorem \ref{teo2dirichlet} for domains in $\mathcal O_{\alpha, r}$.
\begin{teo} \label{propalfa}
Given $\{\Omega_{n}\}_{n \in \N} \subset \mathcal{O}_{\alpha,r_{0}}(D)$ such that $\Omega_{n}\stackrel{H}{\rightarrow}\Omega$. Then, $u^{f}_{\Omega_{n}}\rightharpoonup u^{f}_{\Omega}$ in $W_{0}^{1,p(x)}(D)$.
\end{teo}

\begin{proof}
By Theorem \ref{teoindep}, we can assume that $f=1$ and $u^{1}_{\Omega_{n}}\rightharpoonup u^{*}$ in $W_{0}^{1,p(x)}(D)$.

In order to see that $u^{*}=u^{1}_{\Omega}$, by Theorem \ref{teo1}, it is enough to verify that $u^{*} \in W_{0}^{1,p(x)}(\Omega)$. By Theorem \ref{teocaracterizacion}, it is enough to prove that $\tilde{u^{*}}=0$ $p(x)-$q.e. in $\Omega^{c}$.

As a direct consecuence of Lemma \ref{ppiomax}, $u^{1}_{D}\geq 0$ and $u^{1}_{\Omega_{n}}\geq 0$.

By Lemma \ref{lemawiener}, given $y \in \partial D$, for every $x \notin \Omega$ we have
$$
u^{1}_{D}(x)=|u^{1}_{D}(x)-u^{1}_{D}(y)|\leq M |x-y|^{\delta}\leq M (\diam D)^{\delta}.
$$
By Lemma \ref{ppiomax}, $0\leq u^{1}_{\Omega_{n}}\leq u^{1}_{D}\leq M (\diam D)^{\delta}$. Therefore, $\{u_n\}_{n\in\N}$ is uniformly bounded.

By Lemma \ref{lemawiener}, $\{u_n\}_{n\in\N}$ is uniformly equicontinuous. Therefore, $\{u_n\}_{n\in\N}$ converges uniformily to $u^{*}$.

Given $x \in \Omega^{c}$, since $\Omega_{n}\stackrel{H}{\rightarrow}\Omega$, there is a sequence $x_{n} \in \Omega_{n}^{c}$ such that $x_{n}$ converges to $x$. By uniform convergence, we have that $u_{n}(x_{n})$ converges to $u^{*}(x)$. Since $\sop u_{n}\subset \bar{\Omega}_{n}$, we obtain that $u_{n}(x_{n})=0$ for every $n$ and, therefore, $u^{*}(x)=0$, which completes the proof. 
\end{proof}

\begin{remark}\label{p>N}
If $p_->N$, the same proof can be applied. It is enough to observe that, by Morrey's estimates, $W^{1,p(x)}_0(D)\subset W^{1,p_-}_0(D)\subset C^\alpha(D)$ with $\alpha = 1-N/p_-$.
\end{remark}

Having presented the previous results, the proof of the extension is similar to the one given by \v{S}ver\'ak for $p=2$. We include it for the reader's convenience.

\begin{defi}
Given $l \in \N$ y $\Omega \subset D$, set $\#\Omega$ the number of connected components of $D\setminus \Omega$. 

Set $\mathcal{O}_{l}(D)=\{\Omega \subset\ D \text{ open} \colon \#\Omega\leq l \}$.
\end{defi}

\begin{teo} \label{teosverak}
Given $p \in \PP^{log}(D)$ such that $N-1<p_-$ and $\{\Omega_{n}\}_{n \in \N}\subset \mathcal{O}_{l}(D)$ such that $\Omega_{n}\stackrel{H}{\rightarrow}\Omega$. Then $u^{f}_{\Omega_{n}}\rightharpoonup u^{f}_{\Omega}$ in $W_{0}^{1,p(x)}(D)$.
\end{teo}

\begin{proof}
By Remark \ref{p>N}, we only have to consider the case $N-1<p_-\le N$.

By Theorem \ref{teoindep}, we can assume that $f=1$ and $u_{n}=u^{1}_{\Omega_{n}}\rightharpoonup u^{*}$ en $W_{0}^{1,p(x)}(D)$. 

In order to see that $u^{*}=u^{1}_{\Omega}$, by Theorem \ref{teo1}, it is sufficient to verify that $u^{*} \in W_{0}^{1,p(x)}(\Omega)$.

Set $\bar{D} \setminus \Omega_{n}=F_{n}=F_{n}^{1}\cup F_{n}^{2}\cup...\cup F_{n}^{l}$ where each $F_{n}^{i}$ is compact and connected. Assume that $F_{n}^{j}\stackrel{H}{\rightarrow}F^{j}$ for every $1\leq j\leq l$.

Let us analyze each of the three possibilities. We will find that it is posible to disregard the first two.

(1) If $F^{j}=\emptyset$, then $F_{n}^{j}=\emptyset$ for every $n\geq n_{0}$. Set $J_{0}=\{j\in \{1,\dots,l\} \colon F_{n}^{j}=\emptyset \text{ for j large}\}$.

(2) If $F^{j}=\{x_{j}\}$, set $J_{1}=\{j\in \{1,\dots,l\} \colon F^{j}=\{x_j\} \text{ and } p(x_j)\leq N\}$. Now consider the set $\Omega^{*}=\Omega \setminus \cup_{i \in J_1}\{x_i\}$. Since $\cp_{p(x)}(\{x_i\},D)=0$, we have that $\cp_{p(x)}(\Omega^*,D)=\cp_{p(x)}(\Omega,D)$. Then, by Theorem \ref{teocaracterizacion}, $W_0^{1,p(x)}(\Omega^*)=W_0^{1,p(x)}(\Omega)$. It is enough therefore to verify that $u^* \in W_0^{1,p(x)}(\Omega^*)$.

Set $I=\{1,\dots,l\} \setminus (J_0\cup J_1)$ and consider $\Omega_n^*=D \setminus \cup_{j \in I}F_n^j\stackrel{H}{\rightarrow}\Omega^{*}$.

(3) If, for $j \in I$, $F^{j}$ contains al least two points. Let $a_{j}$ be the distance between them. These points are limits of points from $ F_{n}^{j}$ which we may assume to have a distance at least of $\frac{a_j}{2}$ between them for $n$ large enough.

Given $x \in \partial \Omega_{n}^{*}$ and $j=j(x) \in I$ such that $x \in F_{n}^{j}$, by Corollary \ref{propdefinitivo}, if $a\leq r<\frac{a_{j}}{4}$ for some positive constant $a$, then there is a universal constant $\kappa$ that verifies the following inequality:
$$
\cp_{p(x)}((\Omega_{n}^{*})^{c}\cap \overline{B(x,r)},B(x,2r))\geq \cp_{p(x)}(F_{n}^{j}\cap \overline{B(x,r)},B(x,2r))\geq \kappa>0.
$$
This shows that the open sets $\Omega_{n}^{*}$ belong to $\mathcal{O}_{\alpha,r_{0}}$ with $\alpha=\kappa$ and $r_{0}=\frac{1}{4}\min\{a_{j}\colon j \in I\}$.

Since $\Omega_{n}^{*}\stackrel{H}{\rightarrow}\Omega$, by Theorem \ref{propalfa}, we have that $u^{1}_{\Omega_{n}^{*}}\rightharpoonup u^{1}_{\Omega}$ in $W_{0}^{1,p(x)}(D)$.

On the other hand, since $\Omega_{n}\subset \Omega_{n}^{*}$, by a direct consequence of Lemma \ref{ppiomax} and Proposition \ref{propmonotonia}, we have that $0\leq u^{1}_{\Omega_{n}}\leq u^{1}_{\Omega_{n}^{*}}$. Passing to the limit $n\to\infty$, $0\leq u^{*}\leq u^{1}_{\Omega}$. We conclude then, by Lemma \ref{lemainf}, that $u^{*} \in W_{0}^{1,p(x)}(\Omega)$

If $F^{j}$ contains exactly one point $x_0$, then $p(x_0)>N$ and so $\{x_0\}$ has positive $p(x)-$capacity, the bound from below will be its capacity, which completes the proof.
\end{proof}

\section*{Acknowledgements}
This work was partially supported by Universidad de Buenos Aires under grant UBACYT 20020100100400 and by CONICET (Argentina) PIP 5478/1438.

\def\ocirc#1{\ifmmode\setbox0=\hbox{$#1$}\dimen0=\ht0 \advance\dimen0
  by1pt\rlap{\hbox to\wd0{\hss\raise\dimen0
  \hbox{\hskip.2em$\scriptscriptstyle\circ$}\hss}}#1\else {\accent"17 #1}\fi}
  \def\ocirc#1{\ifmmode\setbox0=\hbox{$#1$}\dimen0=\ht0 \advance\dimen0
  by1pt\rlap{\hbox to\wd0{\hss\raise\dimen0
  \hbox{\hskip.2em$\scriptscriptstyle\circ$}\hss}}#1\else {\accent"17 #1}\fi}
\providecommand{\bysame}{\leavevmode\hbox to3em{\hrulefill}\thinspace}
\providecommand{\MR}{\relax\ifhmode\unskip\space\fi MR }
% \MRhref is called by the amsart/book/proc definition of \MR.
\providecommand{\MRhref}[2]{%
  \href{http://www.ams.org/mathscinet-getitem?mr=#1}{#2}
}
\providecommand{\href}[2]{#2}
\bibliographystyle{plain}
\bibliography{biblio}

\end{document}